\documentclass[10pt]{article}
\usepackage{amsmath}
\usepackage{amssymb,bbm}
\usepackage{enumitem}

\usepackage{tikz}
\usetikzlibrary{arrows}

\usepackage[pdftex,colorlinks=true,linkcolor=blue,unicode,pdfpagelabels]{hyperref}

\newcommand{\ssc}{\hyperref[de:ssc]{standard source condition}}
\newcommand{\svi}{\hyperref[de:svi]{symmetrized variational inequality}}
\newcommand{\hvi}{\hyperref[de:hvi]{homogeneous variational inequality}}
\newcommand{\ivi}{\hyperref[de:ivi]{inhomogeneous variational inequality}}
\newcommand{\nfcr}{\hyperref[de:nfcr]{noise-free convergence rate}}
\newcommand{\ncr}{\hyperref[de:cr]{convergence rate}}
\newcommand{\scr}{\hyperref[de:scr]{spectral tail}}

\newcommand{\R}{\mathbbm R}

\newcommand{\N}{\mathbbm N}

\renewcommand{\d}{\mathrm d}
\newcommand{\argmin}{\mathop{\mathrm{argmin}}}

\newcommand{\abs}[1]{\left| #1 \right|}
\newcommand{\set}[1]{\left\{ #1 \right\}}
\newcommand{\norm}[1]{\| #1 \|}
\newcommand{\inner}[1]{\langle #1 \rangle}

\newcommand{\uad}{u_\alpha^\delta}
\newcommand{\ua}{u_\alpha}
\newcommand{\udag}{u^\dagger}
\newcommand{\yd}{y^\delta}

\usepackage[hyperref,amsmath,thmmarks]{ntheorem}
\usepackage{aliascnt}

\theoremlisttype{all}
\theoremseparator{.}

\newtheorem{theorem}{Theorem}

\newaliascnt{lemma}{theorem}
\newtheorem{lemma}[lemma]{Lemma}
\aliascntresetthe{lemma}

\newaliascnt{corollary}{theorem}

\aliascntresetthe{corollary}

\newaliascnt{proposition}{theorem}
\newtheorem{proposition}[proposition]{Proposition}
\aliascntresetthe{proposition}

\newaliascnt{example}{theorem}
\newtheorem{example}[example]{Example}
\aliascntresetthe{example}

\theorembodyfont{\normalfont}

\newaliascnt{definition}{theorem}
\newtheorem{definition}[definition]{Definition}
\aliascntresetthe{definition}

\theoremstyle{nonumberplain}
\theoremseparator{:}
\theoremheaderfont{\normalfont\itshape}
\theorembodyfont{\normalfont}

\newtheorem{remark}{Remark}

\theoremsymbol{\ensuremath{\square}}
\newtheorem{proof}{Proof}

\allowdisplaybreaks[1]

\begin{document}

\title{Generalized Convergence Rates Results for Linear Inverse Problems in Hilbert Spaces}

\author{%
	Roman Andreev%
	\thanks{
		Johann Radon Institute for Computational and Applied Mathematics (RICAM),
		Altenbergerstrasse 69, A-4040 Linz, Austria
		({\tt roman.andreev@oeaw.ac.at})
	}
	\and
	Peter Elbau%
	\thanks{
		Computational Science Center, University of Vienna,
		Oskar-Morgenstern Platz 1, A-1090 Vienna, Austria
		({\tt peter.elbau@univie.ac.at})
	}
	\and
	Maarten V. de Hoop%
	\thanks{
		Center for Computational and Applied Mathemematics, Purdue University,
		West Lafayette, IN 47907, USA
		({\tt mdehoop@purdue.edu})
	}
	\and
	Lingyun Qiu%
	\thanks{
		Institute for Mathematics and its Applications, University of Minnesota, 
        Minneapolis, MN 55455, USA
        ({\tt qiu.lingyun@ima.umn.edu})
	}
	\and
	Otmar Scherzer%
	\thanks{
		Computational Science Center, University of Vienna and RICAM
		({\tt otmar.scherzer@univie.ac.at})
	}
}
\maketitle

%%%%%%%%%%%%%%%%%%%%%%%%%%%%%%
%%% Abstract
%%%%%%%%%%%%%%%%%%%%%%%%%%%%%%
\begin{abstract}
In recent years, a series of convergence rates conditions for regularization methods
has been developed. Mainly, the motivations for developing novel conditions
came from the desire to carry over convergence rates results from the Hilbert space setting to generalized Tikhonov regularization in Banach spaces. For instance, variational
source conditions have been developed and they were expected to be equivalent to standard source conditions for linear inverse problems in a Hilbert space setting (see Schuster et al
\cite{SchKalHofKaz12}). We show that this expectation does not hold. However, in the standard Hilbert space setting these novel conditions are optimal, which we prove by using
some deep results from Neubauer \cite{Neu97}, and generalize existing convergence rates results. The key tool in our analysis is a novel source condition, which we put into relation
to the existing source conditions from the literature. As a positive by-product, 
convergence rates results
can be proven without spectral theory, which is the standard technique for proving convergence rates for linear inverse problems in Hilbert spaces (see Groetsch \cite{Gro84}).
\end{abstract}

%%%%%%%%%%%%%%%%%%%%%%%%%%%%%%
%%% Introduction
%%%%%%%%%%%%%%%%%%%%%%%%%%%%%%
\section{Introduction}
In this paper we consider for some (not exactly known) data $y\in\mathcal R(L)$ the operator equation
\begin{equation}
\label{eq:op}
Lu = y\,,
\end{equation}
where $L : U \to V$ is a bounded linear operator between two real Hilbert spaces $U$ and $V$,
and
$\mathcal{R}(L)$ denotes its range.

Given some approximate data $\yd\in V$ with $\norm{y-\yd}\le\delta$, the objective is to reconstruct the minimal norm solution $\udag\in U$, that is the element fulfilling
\begin{equation*}
L\udag=y\quad\text{and}\quad \norm{\udag} = \inf \set{\norm{u} : Lu=y}\;.
\end{equation*}
Such a minimal norm solution exists for every $y\in\mathcal R(L)$ and is uniquely defined, see for example \cite[Theorem 2.5]{EngHanNeu96}.

The method of choice for performing this task is Tikhonov regularization, that is to find for arbitrary $\alpha>0$ the regularized solution
\begin{align}\label{e:uad}
\uad := \argmin_{u \in U} \set{ \norm{L u - \yd}^2 + \alpha \norm{u}^2 }\,.
\end{align}
Standard results on Tikhonov regularization guarantee the existence and uniqueness of the minimizer $\uad$ and that 
%$u_\alpha \to u^\dagger$ and 
$\uad$ converges to $u^\dagger$ for an appropriate choice of $\alpha$ depending on $\delta$ as $\delta \searrow 0$, see for instance \cite[Theorem~5.1 and Theorem~5.2]{EngHanNeu96}.

Convergence rates conditions, moreover, guarantee a certain convergence rate
$\norm{\uad - \udag} = \mathcal{O}(f(\delta))$ as $\delta \searrow 0$ 
if again $\alpha$ is chosen to depend in the right way on $\delta$.

Two kinds of such convergence rates conditions have been developed:
\begin{itemize}
\item source conditions \cite{Gro84} and
\item variational source conditions \cite{HofKalPoeSch07,SchGraGroHalLen09,HeiHof09,SchKalHofKaz12}.
\end{itemize}
The goal of this paper is to put the different source conditions into perspective, together with three novel variational source conditions, which are presented here for the first time.
The main results on the relations in between the source conditions are summarized in a table form (cf. Figure \ref{fg:relations}).
Aside from these particular conditions the novelties are to show that these conditions are in fact more general than the classical source conditions,
and they are optimal in the sense that convergence rates of a certain order are only possible if these conditions are satisfied. The argumentation is based on a result
from Neubauer \cite{Neu97}.
Moreover, as a side product, this clarifies some assertion from \cite{SchKalHofKaz12} on the equivalence of standard and variational source conditions.

\section{Relations of Source Conditions in the Standard Setting}\label{sec:source}
The key to obtain convergence rate results for the regularized solution $\uad$, defined in \eqref{e:uad}, of the problem \eqref{eq:op} is to impose conditions on the minimal norm solution $\udag$. In the literature, various kind of such source conditions have been introduced.

\begin{definition}\label{de:source}
Let $U$ and $V$ be real Hilbert spaces, $L:U\to V$ be a bounded linear operator, and $y\in\mathcal R(L)$. Moreover, let $\udag$ denote the minimum-norm solution of the operator equation \eqref{eq:op}. 

Then, we say that the problem fulfills
\begin{itemize}

\item\label{de:ssc}
the standard source condition, see \cite{Gro84}, with the parameter $\nu\in(0,2]$ if
\begin{equation}\label{eq:source_old}
\udag \in \mathcal{R}((L^* L)^{\frac{\nu}{2}});
\end{equation}

\item\label{de:hvi}
the homogeneous variational inequality with the parameter $\nu\in(0,1]$ if there exists a constant $\beta \geq 0$ such that
\begin{equation}\label{eq:source_new_1a}
2\inner{\udag,u} \leq \beta \norm{L u}^\nu \norm{u}^{1-\nu}\quad\text{for every}\quad u \in U;
\end{equation}

\item\label{de:ivi}
the inhomogeneous variational inequality with the parameter $\mu\in(0,1]$, introduced in \cite{HofKalPoeSch07} for the case $\mu=1$ and in \cite{HeiHof09} for general $\mu\in(0,1]$ in the setting of non-linear problems, if there exist constants $\beta\ge0$ and $\gamma\in[0,1)$ such that
\begin{equation}\label{eq:source_new_hof}
2\inner{\udag,u} \leq \beta\norm{L u}^\mu + \gamma\norm{u}^2\quad\text{for every}\quad u \in U;
\end{equation}

\item\label{de:svi}
the symmetrized variational inequality with the parameter $\nu\in(0,2]$ if there exists a constant $\beta\geq 0$ such that
\begin{equation}\label{eq:source_symmetrized}
2\inner{\udag,u} \leq \beta \norm{L^*L u}^{\frac\nu2} \norm{u}^{1-\frac\nu2}\quad\text{for every}\quad u \in U.
\end{equation}

\end{itemize}
\end{definition}

\begin{remark}
Let $\rho \in (0,2]$. The family of variational source conditions,
\[ 2\inner{\udag,u} \le \beta \norm{(L^*L)^{\frac\rho2}u}^{\frac\nu\rho}\norm{u}^{1-\frac\nu\rho},\quad\nu\le\rho,\]
puts the \hvi, the \svi, and the \ssc\ under one umbrella, when we set $\rho=1$, $\rho=2$, and $\rho=\nu$ (see the proof of \autoref{l:first_extended}\ref{en:first_extended_ssc}), respectively. 
However, \autoref{th:neu_follow} and \autoref{th:neu} show that all these variational source conditions with the same parameter $\nu$ and a parameter $\rho>\nu$ are equivalent to each other.
\end{remark}

Note that the \ivi\ is not homogeneous with respect to $u\in U$, as opposed to the other three source conditions.

Let us first discuss the relation between the first three source conditions.
\begin{lemma}
\label{l:first}
	Let $U$ and $V$ be real Hilbert spaces,
	$L : U \to V$ be a bounded linear operator, $y\in\mathcal R(L)$, and $\nu \in (0, 1]$.
	Then, we have that
	\begin{enumerate}%[ref={Lemma \thelemma.{\em\roman*}}]
	\item\label{en:ssc_hvi}
		the \ssc\ for $\nu$ implies the \hvi\ with the same parameter $\nu$,
	\item\label{en:hvi_ivi}
		the \hvi\ with the parameter $\nu$ implies the \ivi\ with the parameter $\mu=\frac{2\nu}{1+\nu}$, and
	\item\label{en:nu_eq_1}
		the \ivi\ with the parameter $\mu=1$ implies the \ssc\ with the parameter $\nu=1$.
	\end{enumerate}
\end{lemma}
\begin{proof}
	Let $\udag$ denote the minimum-norm solution of the operator equation \eqref{eq:op}.
	\begin{enumerate}
	 \item
		If the \ssc\ is fulfilled for some $\nu\in(0,1]$, then there exists an element $\omega\in U$ with $(L^*L)^{\frac\nu2}\omega=\udag$. Using now the interpolation inequality %(\autoref{le:interpolation}),
		\begin{equation}\label{eq:interpolation}
			\|(L^*L)^ru\| \le \|(L^*L)^qu\|^{\frac rq}\|u\|^{1-\frac rq}\quad\text{for all}\quad u\in U,\;0<r\le q,
		\end{equation}
		see for example \cite[Chapter 2.3]{EngHanNeu96}, with $r=\frac\nu2$ and $q=\frac12$, it follows for every $u\in U$ that
	        \[ 2 \inner{\udag, u} = \inner{2\omega,(L^*L)^{\frac{\nu}{2}}u} \le 2\norm{\omega}\norm{(L^*L)^{\frac\nu2}u} \leq 2 \norm{\omega} \norm{L u}^\nu \norm{u}^{1-\nu}, \]
		which is of the form \eqref{eq:source_new_1a} with the parameter~$\nu$.
	\item 
		If $\udag$ fulfills the variational inequality \eqref{eq:source_new_1a} for some parameters $\nu\in(0,1]$ and $\beta\ge0$, then Young's inequality implies for every $u\in U$ that
	       \[ 2 \inner{\udag, u} \le \beta\norm{L u}^{\nu} \norm{u}^{1-\nu} \le \frac{1+\nu}2 \beta^{\frac2{1+\nu}} \norm{L u}^{\frac{2\nu}{1+\nu}} + \frac{1-\nu}2 \norm{u}^2, \]
		so that the \ivi\ with the parameter $\mu=\frac{2\nu}{1+\nu}$ is fulfilled.
	\item
		If $\udag$ fulfills the inequality \eqref{eq:source_new_hof} for $\mu=1$ and some constants $\beta\ge0$ and $\gamma\in[0,1)$, then, by evaluating it at $u=tv$ for arbitrary $v\in U$ and $t>0$, we find in the limit $t\searrow0$ that
		\begin{equation}\label{eq:source_hof_hom}
			2\inner{\udag,v} \le \beta\norm{Lv} = \beta\norm{(L^*L)^{\frac12}v}\quad\text{for every}\quad v\in U.
		\end{equation}

		Now, it can be shown, see \cite[Lemma 8.21]{SchGraGroHalLen09},
		that if $T:U\to U$ is a bounded linear operator,
		then $\udag \in \mathcal{R}(T^*)$ if and only if
		there exists a constant $C > 0$ such that
		$\inner{ \udag, v } \leq C \norm{ T v }$
		for all $v \in U$.

		Thus, with $T=(L^*L)^{\frac12}$, we find that \eqref{eq:source_hof_hom} is equivalent to $\udag\in\mathcal R((L^*L)^{\frac12})$.
	\end{enumerate}
\end{proof}
\begin{remark}
That the \ssc\ for a parameter $\nu\in(0,1]$ implies the \ivi\ with the parameter $\mu=\frac{2\nu}{1+\nu}$ was already realized in \cite{HofYam10}.

The case $\nu = 1$ has been treated in more generality in \cite[Table 3.1]{SchGraGroHalLen09}.
\end{remark}

Thus, the \hvi\ and the \ivi\ cover only the parameter range $\nu\in(0,1]$ compared to the \ssc. However, the \svi\ is an extension of the \ssc\ in the full parameter range $\nu\in(0,2]$, 
as the following lemma shows. 
\begin{lemma}
\label{l:first_extended}
	Let $U$ and $V$ be real Hilbert spaces,
	$L:U\to V$ be a bounded linear operator, and $y\in\mathcal R(L)$.

	Then, we have that
	\begin{enumerate}
	%\item
	%	the \hvi\ with a parameter $\nu\in(0,1]$ implies the \svi\ with the same parameter $\nu$,
	\item\label{en:first_extended_svi}
		the \ssc\ with a parameter $\nu\in(0,2]$ implies the \svi\ with the same parameter $\nu$,
	\item\label{en:first_extended_ssc}
		the \svi\ with the parameter $\nu=2$ is equivalent to the \ssc\ with the parameter $\nu=2$.
	\end{enumerate}
\end{lemma}
\begin{proof}
	Let $\udag$ denote the minimum-norm solution of the problem \eqref{eq:op}.
	\begin{enumerate}
	%\item
	%	The \ssc\ $\udag\in\mathcal R((L^*L)^{\frac\nu2})$ implies that there exists an element $\omega\in U$ with $(L^*L)^{\frac\nu2}\omega=\udag$. Thus, using the interpolation inequality \eqref{eq:interpolation} with $r=\frac\nu2$ and $q=1$, we find for every $u\in U$ that
	%        \[ 2 \inner{\udag, u} = \inner{2\omega,(L^*L)^{\frac{\nu}{2}}u} \le 2\norm{\omega}\norm{(L^*L)^{\frac\nu2}u} \le 2 \norm{\omega} \norm{L^*Lu}^{\frac\nu2} \norm{u}^{1-\frac\nu2}. \]
	\item
		From the inequality \eqref{eq:source_symmetrized} with some parameters $\nu\in(0,1]$ and $\beta\ge0$, we obtain by applying the Cauchy--Schwarz inequality for every $u\in U$ that
		\[ 2\inner{\udag,u} \le \beta\inner{L^*L u,u}^{\frac\nu2} \norm{u}^{1-\nu} \le \beta\norm{L^*Lu}^{\frac\nu2} \norm{u}^{1-\frac\nu2}, \]
		which is the \svi\ with the parameter $\nu$.
	\item
		The \svi\ with the parameter $\nu=2$ states that there exists a constant $\beta\ge0$ so that
		\[ \left<\udag,u\right> \le \beta \norm{L^*Lu} \quad\text{for all}\quad u\in U. \]
		Now, as in the proof of \autoref{l:first}\ref{en:nu_eq_1}, this is equivalent to $\udag\in\mathcal R(L^*L)$, see \cite[Lemma 8.21]{SchGraGroHalLen09}. 
	\end{enumerate}	
\end{proof}

The following two examples 
illustrate that the degree of ill-posedness of the operator $L$ 
is a criterion for equivalency of the different source conditions. 
Finer results,
establishing in particular the equivalence of the source conditions~\eqref{eq:source_new_1a} and~\eqref{eq:source_new_hof}
and the corresponding convergence rates, see \autoref{cor:convergence_rates},
will be derived in \autoref{cor:neu} below.

\begin{example}
\label{ex:general}
	Let $U$ and $V$ be real Hilbert spaces, $\nu \in (0, 1]$, $L : U \to V$ be a bounded linear operator so that $(L^*L)^{\nu/2}$ has closed range, and $y\in\mathcal R(L)$.

	Then, 
	\begin{enumerate}
	\item\label{en:general_ssc}
		the \ssc\ with parameter $\nu$, 
	\item
		the \hvi\ with parameter $\nu$, and
	\item\label{en:general_ivi}
		the \ivi\ with parameter $\mu=\frac{2\nu}{1+\nu}$
	\end{enumerate}
	are equivalent.
\end{example}
\begin{proof}
	In view of \autoref{l:first},
	we only need to show that \ref{en:general_ivi} implies \ref{en:general_ssc}.
	To that end
	recall that,
	if $T : U \to U$ is a bounded linear self-adjoint operator
	then
	its nullspace
	$\mathcal{N}(T)$
	is the orthogonal complement of the range $\mathcal{R}(T)$,
	and
	$U = \overline{\mathcal{R}(T)} \oplus \mathcal{N}(T)$.
	Since the range of $T = (L^*L)^{\nu/2}$ is closed by assumption,
	we have the orthogonal decomposition
	\begin{equation}
	\label{eq:Udec}
		 U = \mathcal{R}((L^*L)^{\frac{\nu}{2}}) \oplus  \mathcal{N}((L^*L)^{\frac{\nu}{2}}).
	\end{equation}
	Observe now that, if ${u} \in  \mathcal{N}((L^*L)^{\frac{\nu}{2}})$
	then
	\begin{equation*}
		\norm{L {u}}^2 = \inner{L {u}, L {u}} = \inner{(L^*L)^{\nu/2}{u}, (L^*L)^{1-\nu/2}{u}} = 0,
	\end{equation*}
	so that $L u = 0$.
	Therefore, if $\udag$ satisfies \eqref{eq:source_new_hof} with some constants $\beta\ge0$ and $\gamma\in[0,1)$,
	then
	\[
	2\inner{\udag,u} \leq  \gamma\norm{u}^2\quad\text{for every}\quad u\in \mathcal{N}((L^*L)^{\nu/2}).
	\]
	Substituting $u$ by $t u$ in the above inequality with $t>0$, we arrive at
	\[
		2t\inner{\udag,u} \leq  t^2 \gamma \norm{u}^2
		\quad\text{for every}\quad
		u \in \mathcal{N}((L^*L)^{\nu/2}),
		\; t > 0.
	\]
	Dividing by $t$ and letting $t$ go to $0$,
	this implies $\inner{\udag, u} = 0$ 
	whenever $u \in \mathcal{N}((L^*L)^{\nu/2})$.
	By the orthogonality of the decomposition \eqref{eq:Udec}
	we have
	$\udag \in \mathcal{R}( (L^*L)^{\nu/2} )$,
	which is \ref{en:general_ssc}.
\end{proof}
\begin{remark}
As in \autoref{ex:general}, one can also show that the \ssc\ and the \svi\ with the same parameter $\nu\in(0,2]$ are equivalent if $(L^*L)^{\frac\nu2}$ has closed range.
\end{remark}

\begin{example}
\label{ex:counter}
	Let $U$ be a real, separable Hilbert space
	with orthonormal basis $\{ \varphi_n \}_{n \in \N}$.
	We define the compact linear operator $L : U \to U$ by $L(\varphi_n) = 2^{-n} \varphi_n$.
	(Note that its range is not closed,
	for the range of a compact operator is closed
	if and only if it is finite-dimensional.)

	Then, for the data 
	\[ y = \sum_{n\ge1}2^{-\frac{3}2 n}\varphi_n\in\mathcal R(L), \]
	the problem \eqref{eq:op} fulfills the \hvi\ with the parameter $\nu=\frac12$, but not the \ssc\ with parameter $\nu=\frac12$.
	In particular, the two source conditions are not equivalent.

	However, the \ssc\ is fulfilled for every parameter $\nu<\frac12$.
\end{example}
\begin{proof}
	The minimum-norm solution $\udag$ can be directly calculated to be
	\begin{align}
	\label{e:udag-counter}
		\udag = L^{-1}y = \sum_{n \geq 1} 2^{-\frac{n}{2}} \varphi_n.
	\end{align}
	Now, since $L$ is self-adjoint by definition so that we have $(L^* L)^{\frac12} = L$, we see that $\udag\notin\mathcal R((L^*L)^{\frac14})$
	because $L^{-\frac{1}{2}} \udag = \sum_{n \geq 1} \varphi_n \notin U$.

	However, we have for every $\nu<\frac12$ that
	$L^{-\nu} \udag = \sum_{n\ge1}2^{n (\nu - \frac12)} \varphi_n$
	is in $U$, and therefore
	$\udag$ is in the range of $(L^* L)^{\frac\nu2}$ for every $\nu<\frac12$.
		
	For $u \in U$ arbitrary we write $u = \sum_{n \geq 1} 2^{\frac{n}{2}} \gamma_n \varphi_n$
	with some $\gamma_n \in \R$. Then
	\begin{equation*}
		\inner{\udag, u} = \sum_{n \geq 1} \gamma_n,
		\quad
		\norm{u}^2 = \sum_{n \geq 1} 2^n \abs{\gamma_n}^2,
		\quad\text{and}\quad
		\norm{L u}^2 = \sum_{n \geq 1} 2^{-n} \abs{\gamma_n}^2.
	\end{equation*}
	Now we can show that the \hvi\ with parameter $\nu=\frac12$ is fulfilled, more precisely, that we have
	\begin{equation}
	\label{eq:S}
	\inner{\udag, u} \leq 2\sqrt2 \norm{u}^{\frac{1}{2}} \norm{L u}^{\frac{1}{2}} \quad\text{for every}\quad u \in U.
	\end{equation}
	Indeed, set $S := \sum_{n \geq 1} |\gamma_n|$ and let $N\in\N$ be such that
	\[ \tfrac12 S \leq A := \sum_{n \leq N} |\gamma_n| \quad\text{and}\quad
		\tfrac12 S \leq B := \sum_{n \geq N} |\gamma_n| . \]
	Observe that,
	using the Cauchy--Schwarz inequality,
	\[ A^2 \leq \sum_{k \leq N} 2^{k} \sum_{n \leq N} 2^{-n} |\gamma_n|^2 \leq (2^{N+1}-1) \norm{L u}^2 , \]
	and
	\[ B^2 \leq \sum_{k \geq N} 2^{-k} \sum_{n \geq N} 2^{n} |\gamma_n|^2 \leq 2^{-N+1} \norm{u}^2 . \]
	Since we have by definition of $S$ that $\inner{\udag, u} \leq S$ and by the choice of $N$ that $S \leq 2 \sqrt{ A B }$,
	the inequality 
	\eqref{eq:S}
	follows.

	This proof is 
	largely from \cite{SE853764}.
	The proof of \autoref{th:neu_follow} below
	is a more elaborate version of
	the same idea.
\end{proof}

\begin{remark}
	In the above proof we noted that
	\begin{align}
	\label{eq:almostnu}
		\udag 
		\in \mathcal{R}((L^*L)^{\frac{\rho}{2}})
		\quad
		\text{for every}\quad \rho \in [0, \nu)
		.
	\end{align}
	This property is a general consequence of
	the variational source condition \eqref{eq:source_new_1a}.
	This follows from 
	\autoref{cor:convergence_rates}
	and
	\cite[Corollary 2.4]{Neu97}.

	However,
	if $\udag$ satisfies \eqref{eq:almostnu},
	it need not satisfy \eqref{eq:source_new_1a}:
	take $\udag$ as in \eqref{e:udag-counter}
	but with $L(\varphi_n) := n^{-2} 2^{-n} \varphi_n$.
	Then \eqref{eq:almostnu} holds for $\nu = 1/2$.
	But
	$\inner{ \udag, \varphi_n } = 2^{-n/2}$
	is not bounded
	in terms of
	$\norm{ L \varphi_n }^{1/2} \norm{ \varphi_n }^{1/2} = n^{-1} 2^{-n/2}$
	uniformly in $n \geq 1$.
\end{remark}

\section{Rates Results without Spectral Theory}
We briefly review the convergence rate results which follow from the introduced source conditions.

\begin{definition}
Let $U$ and $V$ be real Hilbert spaces, $L:U\to V$ be a bounded linear operator, and $y\in\mathcal R(L)$. Moreover, let $\udag$ denote the minimum-norm solution of the operator equation \eqref{eq:op}. 

Then, we say that the problem has
\begin{itemize}
\item\label{de:nfcr}
a noise-free convergence rate of order $\sigma$ if there exists a constant $C>0$ so that the regularized solution
\begin{equation}\label{eq:reg_sol_nf}
\ua = \argmin_{u\in U}\left(\norm{Lu-y}^2+\alpha\norm{u}^2\right),
\end{equation}
fulfills that
\[ \norm{\ua-\udag} \le C\alpha^\sigma\quad\text{for every}\quad\alpha>0, \]
\item\label{de:cr}
a convergence rate of order $\rho$ if there exists a constant $C>0$ so that the regularized solutions
\begin{equation}\label{eq:reg_sol}
\ua(\tilde y) = \argmin_{u\in U}\left(\norm{Lu-\tilde y}^2+\alpha\norm{u}^2\right),\quad\alpha>0,\;\tilde y\in V,
\end{equation}
fulfill for every $\delta>0$ the inequality
\begin{align}
\label{eq:noQdeltarate}
 \sup\set{\inf_{\alpha>0}\norm{\ua(\tilde y)-\udag}\;:\;\tilde y\in V,\;\norm{\tilde y-y}\le\delta}\le C\delta^\rho
 .
\end{align}
\end{itemize}
\end{definition}

The classical convergence results now state that if a problem \eqref{eq:op} fulfils the \ssc\ for some parameter $\nu\in(0,2]$, then it has a \ncr\ of order 
$\frac\nu{1+\nu}$, see \cite[Corollary 3.1.4]{Gro84}.
For $\nu,\mu \in (0,1]$ the same result can be obtained under the weaker source conditions~\eqref{eq:source_new_1a} and~\eqref{eq:source_new_hof}, see \cite{HeiHof09,SchKalHofKaz12}. 
The simple proof is added here for completeness.

\begin{lemma}
\label{le:rate}
Let $L : U \to V$ be a bounded linear operator between two real Hilbert spaces $U$ and $V$, and $y\in\mathcal R(L)$. 
Moreover, let $\udag$ denote the minimum-norm solution of the problem \eqref{eq:op} and assume that it fulfils the \ivi\ \eqref{eq:source_new_hof} for some parameters $\mu \in (0, 1]$, $\beta\ge0$, and $\gamma\in(0,1)$.

Then, for every choice of $\yd\in V$ with $\norm{\yd-y}\le\delta$ for some $\delta>0$ and every $\alpha>0$, the corresponding regularized solution
\[ \uad = \argmin_{u\in U}\left(\norm{Lu-\yd}^2+\alpha\norm{u}^2\right) \]
satisfies
\begin{equation}\label{eq:distance_to_sol}
\norm{\uad-\udag}^2\le \frac2{1-\gamma}\,\frac{\delta^2}\alpha+\frac{\beta^{\frac2{2-\mu}}(2-\mu)}{2(1-\gamma)}\,\alpha^{\frac\mu{2-\mu}}.
\end{equation}
\end{lemma}

\begin{proof}
From the definition of the minimizer $\uad$, it follows that
\[ \norm{L \uad-\yd}^2 +\alpha \norm{\uad}^2 \leq \delta^2 +\alpha \norm{\udag}^2. \]
This inequality together with the variation inequality \eqref{eq:source_new_hof} yields
\begin{align*}
\norm{L \uad-\yd}^2 +\alpha \norm{\uad-\udag}^2 
&\le \delta^2 + 2 \alpha \inner{\udag,\udag-\uad} \\
&\le \delta^2+\alpha\beta\norm{L (\uad -\udag)}^\mu+\alpha\gamma\norm{\uad -\udag}^2.
\end{align*}
Now, observing that
\[ \frac12 \norm{L(\uad - \udag)}^2 - \delta^2 \leq \norm{L \uad -\yd}^2, \]
which is a consequence of the triangle inequality and the fact that $a \leq b+c$ implies $a^2 \leq 2(b^2+c^2)$,
we further find that
\[ \frac12 \norm{L(\uad - \udag)}^2 + \alpha(1-\gamma)\norm{\uad-\udag}^2 \le 2\delta^2 + \alpha\beta\norm{L (\uad -\udag)}^\mu. \]
Applying then Young's inequality to the last term, we end up with
\[ \frac12 \norm{L(\uad - \udag)}^2 + \alpha(1-\gamma)\norm{\uad-\udag}^2 \le 2\delta^2 + \frac{2-\mu}2(\alpha\beta)^{\frac2{2-\mu}}+\frac\mu2\norm{L (\uad -\udag)}^2, \]
which in particular implies \eqref{eq:distance_to_sol}.
\end{proof}

\begin{proposition}\label{cor:convergence_rates}
Assume that $L: U \to V$ is a bounded linear operator between two real Hilbert spaces $U$ and $V$, and $y\in\mathcal R(L)$. 

Then, if the problem \eqref{eq:op} fulfills the \ivi\ with the parameter $\mu=\frac{2\nu}{1+\nu}$ for some $\nu\in(0,1]$, it has
\begin{enumerate}
\item
a \nfcr\ of order $\frac\nu2$ and
\item
a \ncr\ of order $\frac\nu{1+\nu}$.
\end{enumerate}
\end{proposition}

\begin{proof}
Let $\udag$ be the minimal-norm solution of \eqref{eq:op}.
\begin{enumerate}
\item
In the noise free case, \autoref{le:rate} with $\delta=0$ and $\mu=\frac{2\nu}{1+\nu}$ directly implies for the regularized solution $\ua$ defined by \eqref{eq:reg_sol_nf} the inequality
\[ \norm{\ua-\udag} \le C\alpha^{\frac\mu{2(2-\mu)}} = C\alpha^{\frac\nu2}\quad\text{for all}\quad\alpha>0 \]
for some constant $C>0$.
\item
In the noisy case, \autoref{le:rate} yields for arbitrary $\delta>0$ and data $\tilde y\in V$ with $\norm{\tilde y-y}\le\delta$ the inequality
\begin{equation}\label{eq:deltarate}
\inf_{\alpha>0}\norm{\ua(\tilde y)-\udag}^2 \le \norm{u_{\delta^{2-\mu}}(\tilde y)-\udag}^2 \le C\delta^\mu = C\delta^{\frac{2\nu}{1+\nu}}
\end{equation}
for some constant $C>0$. Here, $\ua(\tilde y)$ denotes the regularized solution~\eqref{eq:reg_sol}.
\end{enumerate}
\end{proof}

\begin{remark}
Because of \autoref{l:first}, the \hvi\ with a parameter $\nu\in(0,1]$ therefore also implies a \nfcr\ of order~$\frac\nu2$ and a \ncr\ of order~$\frac\nu{1+\nu}$.
\end{remark}

\section{On converse results of Neubauer}
In this section we go deeper into the results of Neubauer \cite{Neu97}.
In the Hilbert space setting, Neubauer characterized the minimum-norm solution for 
which the problem has a convergence rate of order $\frac\nu{\nu+1}$ for some $\nu\in(0,2)$
in terms of its spectral tail.
(Note that Neubauer writes $2 \nu$ where we write $\nu$.)

\begin{definition}\label{de:scr}
Let $U$ and $V$ be real Hilbert spaces,	$L:U\to V$ be a bounded linear operator, and $y\in\mathcal R(L)$. We say that the minimum-norm solution $\udag$ of the problem \eqref{eq:op} has spectral tail of order $\nu$ if there exists a constant $C>0$ so that
\begin{equation}\label{eq:NeuT21}
\norm{E_{[0,\lambda]}\udag}^2 \le C^2\lambda^\nu\quad\text{for all}\quad \lambda\ge0,
\end{equation}
where $A\mapsto E_A$ denotes the (projection-valued) spectral measure of the operator~$L^*L$.
\end{definition}

\begin{proposition}
\label{th:neu}
	Let $L:U\to V$ be a bounded linear operator between two real Hilbert spaces $U$ and $V$, and $y\in\mathcal R(L)$.

	Then, for every $\nu\in(0,2)$, it is equivalent for the problem \eqref{eq:op} that
	\begin{enumerate}
	\item\label{en:neu_nfcr}
		it has a \nfcr\ of order $\frac\nu2$,
	\item\label{en:neu_ncr}
		it has a \ncr\ of order $\frac\nu{\nu+1}$, and
	\item\label{en:neu_scr}
		its minimum-norm solution has a \scr\ of order $\nu$.
	\end{enumerate}
\end{proposition}

\begin{proof}
	Neubauer showed in \cite[Theorem 2.1]{Neu97} that the condition \ref{en:neu_nfcr} is equivalent to \ref{en:neu_scr},
	and proved in \cite[Theorem 2.6]{Neu97} that \ref{en:neu_scr} is equivalent to the fact that there exists a constant $C\ge0$ so that
	\begin{equation}
	\label{eq:Qdeltarate}
		\sup\left\{\inf_{\alpha>0}\norm{\ua(\tilde{y}) - \udag}\; :\;\tilde  y\in V,\;\norm{Q(\tilde{y}-y)} \leq \delta\right\} \le C\delta^{\frac{\nu}{\nu+1}}
	\end{equation}
	for every $\delta\ge0$, where $Q$ denotes the orthogonal projection onto the range $\overline{\mathcal R(L)}$ and the regularized solution $u_\alpha(\tilde y)$ is defined by \eqref{eq:reg_sol}.
	
	It therefore only remains to show that \eqref{eq:Qdeltarate} is equivalent to a \ncr\ of order $\frac\nu{\nu+1}$.
	
	It is clear that \eqref{eq:Qdeltarate} implies such a \ncr, since the supremum in the definition \eqref{eq:noQdeltarate} of the \ncr\ is taken over a smaller set than in \eqref{eq:Qdeltarate}.

	For the other direction, we define for arbitrary $\tilde{y}\in V$ with 
	$\norm{ Q ( \tilde{y} - y ) } \leq \delta$,
	the element
	$\hat{y} := y + Q ( \tilde{y} - y )$.
	Then,
	\[ \norm{ \hat{y} - y } = \norm{ Q( \tilde{y} - y ) } \leq \delta, \]
	and the optimality conditions for the regularized solutions $\ua(\tilde y)$ and $\ua(\hat y)$ yield
	\[ \ua(\tilde{y}) - \ua(\hat{y}) = (\alpha I + L^* L)^{-1} L^* (\tilde{y} - \hat{y}). \]
	Since now $L^* = L^* Q$, we have 
	$L^* (\tilde{y} - \hat{y}) = L^* Q (\tilde{y} - \hat{y}) = 0$
	and therefore, $\ua(\tilde{y}) = \ua(\hat{y})$.
	As $\tilde{y}$ was arbitrary subject to $\norm{ Q ( \tilde{y} - y ) } \leq \delta$,
	condition \eqref{eq:noQdeltarate} with $\rho=\frac\nu{\nu+1}$ implies \eqref{eq:Qdeltarate}.
\end{proof}

Neubauer \cite{Neu97} also gave a counterexample to show that the \ssc\ with parameter $\nu\in(0,2)$, which implies the three equivalent conditions of \autoref{th:neu}, is not equivalent to them, 
see also \autoref{ex:counter}.

However, we will show in the following that the \hvi\ with parameter $\nu\in(0,1)$ and the \ivi\ with parameter $\mu=\frac{2\nu}{1+\nu}$ are indeed equivalent to the conditions of \autoref{th:neu}.

\begin{proposition}
\label{th:neu_follow}
	Let $L:U\to V$ be a bounded linear operator between two real Hilbert spaces $U$ and $V$, and let $y\in\mathcal R(L)$.

	Then for arbitrary $\nu\in(0,2)$ and $\rho>\nu$ the conditions that 
	\begin{enumerate}
	\item\label{en:neu_follow_spectral_tail}
		the maximum-norm solution $\udag$ of the problem \eqref{eq:op} has a \scr\ of order $\nu$ and
	\item\label{en:neu_follow_variational}
		there exists a constant $\beta\ge0$ so that
		\begin{equation}\label{eq:neu_follow_variational}
			2\inner{\udag,u} \le \beta\norm{(L^*L)^{\frac\rho2}u}^{\frac\nu\rho}\norm{u}^{1-\frac\nu\rho}\quad\text{for all}\quad u\in U
		\end{equation}
	\end{enumerate}
	are equivalent.
\end{proposition}
\begin{proof}
	We first show that \ref{en:neu_follow_spectral_tail} implies \ref{en:neu_follow_variational}.

	Let $A\mapsto E_A$ denote the (projection-valued) spectral measure of $L^*L$. For arbitrary $u\in U$, we define the signed measure $A\mapsto\mu_{\udag,u}(A)=\inner{E_A\udag,u}$ 
	and set for $\lambda \in [0,\infty]$
	\begin{align}
		A_\lambda := |\mu_{\udag,u}|([0,\lambda])
		\quad\text{and}\quad
		B_\lambda := |\mu_{\udag,u}|([\lambda,\infty)),
	\end{align}
	where $|\mu_{\udag,u}|$ denotes the variation of the measure $\mu_{\udag,u}$.

	Let now $\Lambda := \inf \{ \lambda \geq 0 : A_\lambda \geq \frac12 A_\infty \}$.
	Then, since $\lambda \mapsto A_\lambda$ is right-continuous, there holds $A_\Lambda \geq \frac12 A_\infty$. 
	Moreover, since $\Lambda$ is minimal and $\lambda\mapsto B_\lambda$ is left-continuous, it also follows that $B_\Lambda \geq \frac12 A_\infty$.

	We now estimate $A_\Lambda$ with the inequality \eqref{eq:Lambda} with $T=L^*L$ and $\rho=0$, which yields
	\[ A_\Lambda = |\mu_{\udag,u}|([0,\Lambda]) \le \norm{E_{[0,\Lambda]}\udag}\norm{u}. \]
	If the \scr\ of $\udag$ has order $\nu$, we have a constant $C>0$ so that $\norm{E_{[0,\lambda]}\udag}\le C\lambda^{\frac\nu2}$ and thus
	\begin{equation}
	\label{eq:C_estimate}
	A_\Lambda \le C\norm{u}\Lambda^{\frac\nu2}.
	\end{equation}
	For $B_\Lambda$, we also use the inequality \eqref{eq:Lambda} with $T=L^*L$, and get for arbitrary $\rho\in\R$,  the upper bound
	\[ B_\Lambda = |\mu_{\udag,u}|([\Lambda,\infty)) \le \norm{(L^*L)^{\frac\rho2}u}\left(\int_{[\Lambda,\infty)}\frac1{\lambda^\rho}\d\mu_{\udag,\udag}(\lambda)\right)^{\frac12}. \]
	Choosing now $\rho>\nu$, we can estimate the integral with \eqref{eq:est_integral} (using that the measure $\mu_{\udag,\udag}$ satisfies $\mu_{\udag,\udag}([0,\lambda])=\norm{E_{[0,\lambda]}\udag}^2\le C^2\lambda^\nu$) 
	and find 
	\[ B_\Lambda \le \frac C{\sqrt{1-\frac\nu\rho}}\norm{(L^*L)^{\frac\rho2}u}\Lambda^{\frac{\nu-\rho}2}. \]

	Therefore, recalling that $\Lambda$ was chosen so that $A_\Lambda\ge\frac12A_\infty$ and $B_\Lambda\ge\frac12A_\infty$, we have
	\[ 2\inner{\udag,u} \le 2A_\infty \le 4A_\Lambda^{1-\frac\nu\rho}B_\Lambda^{\frac\nu\rho} \le \frac{4C}{(1-\frac\nu\rho)^{\frac\nu{2\rho}}}\norm{u}^{1-\frac\nu\rho}\norm{(L^*L)^{\frac\rho2}u}^{\frac\nu\rho}, \]
	which is the condition \ref{en:neu_follow_variational}.

	For the other direction, we remark that from the inequality \eqref{eq:neu_follow_variational} for some constant $\beta\ge0$, we find for every $\lambda\ge0$ that
	\[ \norm{E_{[0,\lambda]}\udag}^2 = \inner{E_{[0,\lambda]}\udag,\udag} \le \frac\beta2\norm{(L^*L)^{\frac\rho2}E_{[0,\lambda]}\udag}^{\frac\nu\rho}\norm{E_{[0,\lambda]}\udag}^{1-\frac\nu\rho}. \]
	Now, since $E$ is the spectral measure of $L^*L$, we have that $\norm{(L^*L)^{\frac\rho2}E_{[0,\lambda]}\udag} \le \lambda^{\frac\rho2}\norm{E_{[0,\lambda]}\udag}$, 
	see for example \cite[Chapter~X.2.9, Corollary~9]{DunSchw63},
	and we therefore obtain that 
	\[ \norm{E_{[0,\lambda]}\udag}^2 \le \frac\beta2\lambda^{\frac\nu2}\norm{E_{[0,\lambda]}\udag}, \]
	which concludes the proof.
\end{proof}

\begin{remark}
In fact, it can be seen from this proof that condition \ref{en:neu_follow_variational} in \autoref{th:neu_follow} also implies condition \ref{en:neu_follow_spectral_tail} in the case $\rho=\nu=2$, which corresponds to the result that the \ssc\ for $\nu=2$ yields a \nfcr\ of order $1$.
\end{remark}

Finally, we can summarize all the statements in an equivalence result between the different source conditions and convergence rates.

\begin{theorem}
\label{cor:neu}
	Let $L : U \to V$ be a bounded linear operator between two real Hilbert spaces $U$ and $V$ and $y\in\mathcal R(L)$.

	Then, for every $\nu\in(0,2)$, it is equivalent for the problem \eqref{eq:op} that
	\begin{enumerate}
	\item\label{en:main_svi}
		it fulfils the \svi\ with parameter $\nu$,
	\item\label{en:main_nfcr}
		it has a \nfcr\ of order $\frac\nu2$,
	\item\label{en:main_ncr}
		it has a \ncr\ of order $\frac\nu{\nu+1}$, 
	\item\label{en:main_scr}
		its minimum-norm solution $\udag$ has a \scr\ of order $\nu$,
	\end{enumerate}
	and if $\nu\in(0,1)$ these are additionally equivalent to 
	\begin{enumerate}
	\addtocounter{enumi}{4}
	\item\label{en:main_hvi}
		the \hvi\ with parameter $\nu$ and
	\item\label{en:main_ivi}
		the \ivi\ with parameter $\mu=\frac{2\nu}{1+\nu}$.
	\end{enumerate}
\end{theorem}
\begin{proof}
We already know from \autoref{th:neu} that \ref{en:main_nfcr}, \ref{en:main_ncr}, and \ref{en:main_scr} are equivalent conditions. Moreover, we know from \autoref{l:first} that \ref{en:main_hvi} implies \ref{en:main_ivi}, and from \autoref{cor:convergence_rates} that \ref{en:main_ivi} implies \ref{en:main_nfcr} and \ref{en:main_ncr}.

Now, \autoref{th:neu_follow} with $\rho=1$ shows that \ref{en:main_scr} implies \ref{en:main_hvi}, which proves the equivalence of all conditions but \ref{en:main_svi}.

And finally, the equivalence of \ref{en:main_svi} and \ref{en:main_scr} follows directly from \autoref{th:neu_follow} with $\rho=2$.

\end{proof}

We briefly comment on the case $\nu = 1$. We have already seen in \autoref{l:first} that in this case the \ssc, the \hvi\ and the \ivi\ (all with the parameter $1$) are equivalent. Moreover, because of \autoref{cor:convergence_rates} they also imply all the conditions of \autoref{th:neu}. However, the converse is not true.
\begin{example}
	Let $U$ be a real, separable Hilbert space with orthonormal basis $\{ \varphi_n \}_{n \in \N}$.
	We define the compact linear operator $L : U \to U$ by $L\varphi_n = n^{-\frac12}\varphi_n$.

	Then, for the data 
	\[ y = \sum_{n\ge1}n^{-\frac32}\varphi_n\in\mathcal R(L), \]
	the minimum-norm solution $\udag$ of problem \eqref{eq:op} has \scr\ of order $1$, but the problem does not fulfill the \ivi\ with the parameter $\mu=1$.
\end{example}
\begin{proof}
	We see that the minumum norm solution $\udag$ is explicitly given by
	\[ \udag = \sum_{n \geq 1} n^{-1} \varphi_n \]
	and thus has a \scr\ of order $1$:
	\[ \norm{E_{[0,\lambda]}\udag}^2 = \sum_{n \geq \lambda^{-2}} |\inner{ \udag, \varphi_n }|^2 = \sum_{n \geq\lambda^{-2}} n^{-2} \le C \lambda^2 \]
	for all $\lambda\in[0,\norm{L}]$ for some constant $C>0$, where $A\mapsto E_A$ again denotes the spectral measure of $L^*L$.

	However, for $u_N=N^{-1}\sum_{n\le N}\varphi_n$, we find that
	\[ \frac{ 2 \inner{ \udag, u_N } }{ \norm{ L u_N } + \norm{ u_N }^2 } = \frac{ 2 N^{-1} H_N }{ N^{-1} H_N^{1/2} + N^{-2} \sum_{n \leq N} 1 } \geq H_N^{1/2}, \]
	where
	$H_N := \sum_{n \leq N} n^{-1}$ denotes the $N$-th harmonic number.
	Because $H_N \to \infty$ as $N \to \infty$, the \ivi\ with parameter $\mu=1$ cannot be satisfied.
\end{proof}

The condition \eqref{eq:source_new_1a} seems to be the natural condition for convergence rates.
It is a necessary and sufficient condition for the rate
$\mathcal{O}(\alpha^{\frac{\nu}{2}})$,
while the standard range condition \eqref{eq:source_old} leaves a small gap.

\autoref{cor:neu} guarantees that the variational source conditions are optimal conditions for convergence rates.

%%%%%%%%%%%%%%%%%%%%%%%%%%%%%%
%%% Figure: Relations between results
%%%%%%%%%%%%%%%%%%%%%%%%%%%%%%
\begin{figure}[ht]
\begin{center}
\begin{tikzpicture}[x=1cm,y=1cm,font=\scriptsize]
\draw[rounded corners=1mm] (0,0) rectangle (3,1) node[pos=0.5,xshift=0.1cm] {\parbox{2.8cm}{Standard\\source condition\\with parameter $\nu$}};
%\draw[implies-,double,densely dashed] (3.1,0.75) -- (3.9,0.75) node[above,yshift=0.2cm,pos=0.5] {\parbox{2.8cm}{\centering\tiny only for $\nu=1$, see \cite{SchKalHofKaz12}, \\ or if $\mathcal R(L^*L)$ is closed, see \autoref{ex:general}}};
\draw[-implies,double] (3.1,0.4) -- (3.9,0.4) node[below,pos=0.5,yshift=-0.35cm] {\parbox{2.8cm}{\centering\tiny for $\nu\in(0,1]$,\\see \autoref{l:first}\ref{en:ssc_hvi}}};

\draw[rounded corners=1mm] (4,0) rectangle (7,1) node[pos=0.5,xshift=0.1cm] {\parbox{2.8cm}{Homogeneous\\ variational inequality\\with parameter $\nu$}};

%\draw[implies-,double] (7.1,0.75) -- (7.9,0.75) node[above,yshift=0.2cm,pos=0.5] {\parbox{2cm}{\centering\tiny\autoref{cor:neu} for $\nu\in(0,1)$ and \\ \cite{SchKalHofKaz12} for $\nu=1$}};
\draw[-implies,double] (7.1,0.4) -- (7.9,0.4) node[below,pos=0.5,yshift=-0.35cm] {\parbox{2.8cm}{\centering\tiny for $\nu\in(0,1]$,\\see \autoref{l:first}\ref{en:hvi_ivi}}};
\draw[rounded corners=1mm] (8,0) rectangle (11,1) node[pos=0.5,xshift=0.1cm] {\parbox{2.8cm}{Inhomogeneous\\ variational inequality\\with parameter $\frac{2\nu}{\nu+1}$}};

\draw[rounded corners=1mm] (8,-3) rectangle (11,-4) node[pos=0.5,xshift=0.1cm] {\parbox{2.8cm}{Convergences rates\\of order $\frac\nu{\nu+1}$}};
\draw[-implies,double] (9,-0.1) -- (6.5,-2.9) node[right,pos=0.5,xshift=0.1cm] {\parbox{2cm}{\centering\tiny for $\nu\in(0,1]$,\\see \autoref{cor:convergence_rates}}};
\draw[-implies,double] (10,-0.1) -- (10,-2.9);% node[right,pos=0.5] {\parbox{2cm}{\centering\tiny only for $\nu\in(0,1]$,\\see \autoref{cor:convergence_rates}}};

\draw[implies-implies,double] (7.1,-3.5) -- (7.9,-3.5) node[below,pos=0.5,yshift=-0.45cm] {\parbox{2.8cm}{\centering\tiny for $\nu\in(0,2)$, \\see \autoref{th:neu} and \cite{Neu97}}};
\draw[rounded corners=1mm] (4,-3) rectangle (7,-4) node[pos=0.5,xshift=0.1cm] {\parbox{2.8cm}{\raggedright Noise-free convergence rates of order $\frac\nu2$}};

\draw[implies-implies,double] (3.1,-3.5) -- (3.9,-3.5) node[below,pos=0.5,yshift=-0.45cm] {\parbox{2.8cm}{\centering\tiny for $\nu\in(0,2)$, \\see \autoref{th:neu} and \cite{Neu97}}};
\draw[rounded corners=1mm] (0,-3) rectangle (3,-4) node[pos=0.5,xshift=0.1cm] {\parbox{2.8cm}{\raggedright Spectral characterization\\ of order $\nu$}};

\draw[implies-implies,double] (5,-0.1) -- (2.5,-2.9) node[right,pos=0.5,xshift=0.2cm] {\parbox{2cm}{\centering\tiny for $\nu\in(0,1)$,\\ see \autoref{th:neu_follow}\\ with $\rho=1$}};

\draw[-implies,double,densely dashed] (7.9,0.75) to[out=150,in=30] (3.1,0.75);
\draw (5.5,2.1) node {\parbox{2.8cm}{\centering\tiny for $\nu=1$, \\ see \autoref{l:first}\ref{en:nu_eq_1}; \\ or for $\nu\in(0,1]$ \\ if $\mathcal R(L^*L)$ is closed, \\ see \autoref{ex:general}}};

\draw[rounded corners=1mm](0,-1) rectangle (3,-2) node[pos=0.5,xshift=0.1cm] {\parbox{2.8cm}{Symmetrized \\ variational imequality \\ with parameter $\nu$}};

\draw[implies-implies,double] (1.5,-2.1) -- (1.5,-2.9) node[pos=0.5,left] {\parbox{2cm}{\centering\tiny for $\nu\in(0,2)$,\\ see \autoref{th:neu_follow}\\ with $\rho=2$}};
\draw[-implies,double] (0.5,-0.1) -- (0.5,-0.9) node[pos=0.5,left] {\parbox{1.6cm}{\centering\tiny for $\nu\in(0,2]$,\\ see \autoref{l:first_extended}\ref{en:first_extended_svi}}};
\draw[implies-implies,double,densely dashed] (2.3,-0.1) -- (2.3,-0.9) node[pos=0.5,left,xshift=0.1cm] {\parbox{1.6cm}{\centering\tiny for $\nu=2$, see\\ \autoref{l:first_extended}\ref{en:first_extended_ssc}}};

\end{tikzpicture}
\end{center}
\caption{Relation between the different source conditions and the convergence rate results.}
\label{fg:relations}
\end{figure}
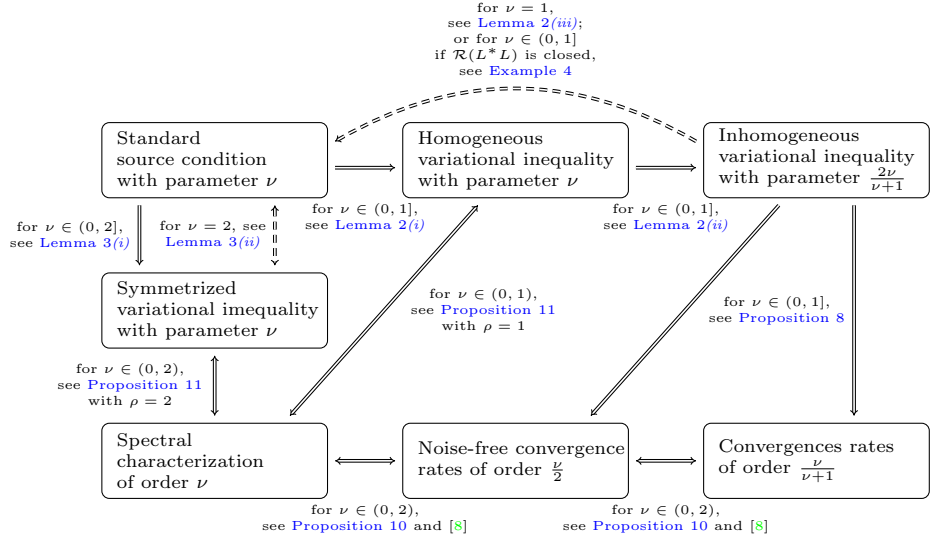

%%%%%%%%%%%%%%%%%%%%%%%%%%%%%%
%%% Acknowledgment
%%%%%%%%%%%%%%%%%%%%%%%%%%%%%%
\section*{Acknowledgment}
This work has been supported by the Austrian Science Fund (FWF)
within the national research network Geometry + Simulation
(project S11704, Variational Methods for Imaging on Manifolds). 
This research was supported in part by National Science Foundation grant CMG DMS-1025318, and in part by the members of the Geo-Mathematical Imaging Group at Purdue University. 

\section*{Conclusion}
In this paper we have developed a series of novel variational source conditions as alternatives to classical source conditions to prove convergence rates results for 
Tikhonov regularization in an Hilbert space setting. In many cases the new source conditions provide optimal convergence rates, opposed to the standard source conditions. 
The interplay between various source conditions and convergence rates is developed in detail and summarized in Table \ref{fg:relations}. As a side product we could clarify 
an open question in \cite{SchKalHofKaz12}. An open question is of course how these results can be generalized to non-linear ill--posed problems, to Banach spaces or 
general topological spaces, and to other regularization methods.

%%%%%%%%%%%%%%%%%%%%%%%%%%%%%%
%%% Appendix
%%%%%%%%%%%%%%%%%%%%%%%%%%%%%%
\section*{Appendix}
\begin{lemma}
\label{l:Lambda}
	Let $T:U\to U$ be a self-adjoint, non-negative definite, bounded linear operator on a real Hilbert space $U$, $\udag, u \in U$, and $\Lambda \geq 0$.
	We denote with $A\mapsto E_A$ the projection-valued spectral measure of $T$ and define for all $v,w\in U$ the signed measure $A\mapsto\mu_{v,w}(A)=\inner{E_Av,w}$.
	
	Then, we have for all $0\le a\le b$ and every $\rho\in\R$ that	
	\begin{equation}\label{eq:Lambda}
	|\mu_{\udag,u}|([a,b]) \le \left(\int_{[a,b]} \lambda^{-\rho}\d\mu_{\udag,\udag}(\lambda)\right)^{\frac12} \left(\int_{[a,b]}\lambda^\rho\d\mu_{u,u}(\lambda)\right)^{\frac12}.
	\end{equation}
\end{lemma}
\begin{proof}
	Using the spectral representation theorem, see for instance \cite[Chapter~X.5.3, Corollary~4]{DunSchw63},
	we may assume that 
	the operator $T$
	is a multiplication operator $u \mapsto m u$, 
	where $m \geq 0$ is a bounded measurable function
	on a measure space $(\Omega, \Sigma, \mu)$
	and $U$ is the Lebesgue space $U = L^2(\Omega;\mu)$.
	In this case,
	the spectral measure of $T$ 
	is given by
	$\inner{ E_A \udag, u } = \int_{m^{-1}(A)} \udag u \d\mu$.
	Therefore, for every $\rho\in\R$,
	we have the representation
	\[ \int_{[a,b]}\lambda^\rho\d|\mu_{\udag,u}|(\lambda) = \int_{m^{-1}([a,b])} m^\rho |\udag|\, |u|\,\d\mu. \]
	
	Thus we can estimate with the Cauchy--Schwarz inequality for arbitrary $\rho\in\R$:
	\begin{align*}
	|\mu_{\udag,u}|([a,b]) &= \int_{m^{-1}([a,b])}|\udag|\,|u|\,\d\mu(\lambda) \\
	&\leq \left(\int_{m^{-1}([a,b])} m^{-\rho} |\udag|^2 \d\mu\right)^{\frac12} \left(\int_{m^{-1}([a,b])} m^{\rho} |u|^2 \d\mu\right)^{\frac12} \\
	&= \left(\int_{[a,b]} \lambda^{-\rho}\d\mu_{\udag,\udag}(\lambda)\right)^{\frac12}\left(\int_{[a,b]} \lambda^\rho\d\mu_{u,u}(\lambda)\right)^{\frac12}.
	\end{align*}
\end{proof}

\begin{lemma}
\label{l:reverse}
	Let $\mu$ be a non-negative finite Borel measure on $\R$ 
	with compact support in $[0, \infty)$.
	Let $0 \leq \nu < \rho$.
	Suppose that there exists a constant $C > 0$
	such that
	$\mu([0, \lambda)) \leq C \lambda^\nu$ for all $\lambda \geq 0$.
	Then 
	\begin{equation}\label{eq:est_integral}
	\int_{[\Lambda,\infty)} \lambda^{-\rho}\d\mu(\lambda) \le C \frac{\rho}{\rho-\nu} \Lambda^{\nu-\rho} \quad\text{for all}\quad \Lambda > 0.
	\end{equation}
\end{lemma}
\begin{proof}
	For $\lambda \geq 0$ define
	$I(\lambda) := \mu([0, \lambda))$
	and
	$g(\lambda) := \lambda^{-\rho}$.
	Then we can write the above integral 
	as a Stieltjes integral,
	and apply integration by parts,
	\begin{align*}
		\int_{[\Lambda,\infty)} \lambda^{-\rho}\d\mu(\lambda)
	& =
		\int_\Lambda^\infty g(\lambda) \d I(\lambda)
	\\
	& =
		\left.
			g(\lambda) I(\lambda) 
		\right|_{\lambda = \Lambda}^{\lambda = \infty}
		-
		\int_\Lambda^\infty I(\lambda) \d g(\lambda)
		.
        \end{align*}
	
	Because $\mu$ is finite and $g(\infty)=0$ it follows that $\left.
			g(\lambda) I(\lambda) 
		\right|_{\lambda = \Lambda}^{\lambda = \infty} \leq 0$. Therefore, and taking into account that 
		$g$ is smooth, it follows that
	\[
		\int_{[\Lambda,\infty)} \lambda^{-\rho}\d\mu(\lambda)
	\leq -
		\int_\Lambda^\infty I(\lambda) \d g(\lambda)
	 = -
		\int_\Lambda^\infty I(\lambda) g'(\lambda) \d\lambda. \]

		Then we use the assumption $I(\lambda) \leq C \lambda^\nu$ and the monotonicity of $g$, to see that
	\[ \int_{[\Lambda,\infty)} \lambda^{-\rho} \d\mu(\lambda) \leq - C \int_\Lambda^\infty \lambda^\nu g'(\lambda) \d\lambda = C \frac{\rho}{\rho - \nu} \Lambda^{\nu - \rho} . \]
	Thus the assertion is proved.
\end{proof}

%%%%%%%%%%%%%%%%%%%%%%%%%%%%%%
%%% Bibliography
%%%%%%%%%%%%%%%%%%%%%%%%%%%%%%

\end{document}